\documentclass[10pt, reqno]{amsart}

\usepackage{amsmath, amsthm, amssymb}
\usepackage[colorlinks=true,urlcolor=blue,linkcolor=blue,citecolor=blue]{hyperref}
\usepackage{enumerate}
\usepackage{geometry}
\usepackage{graphicx}
\usepackage{cleveref}
\usepackage{mathdots}

\makeatletter
\newtheorem*{rep@theorem}{\rep@title}
\newcommand{\newreptheorem}[2]{%
\newenvironment{rep#1}[1]{%
 \def\rep@title{#2 \ref{##1}}%
 \begin{rep@theorem}}%
 {\end{rep@theorem}}}
\makeatother

\newtheorem{theorem}{Theorem}[section]
\newreptheorem{theorem}{Theorem}

\newtheorem{lemma}[theorem]{Lemma}

\newtheorem{conjecture}[theorem]{Conjecture}

\newtheorem{corollary}[theorem]{Corollary}

\newcommand{\N}{\mathbb{N}}
\newcommand{\Z}{\mathbb{Z}}

\newcommand{\C}{\mathbb{C}}

\title{Rankin-Selberg gamma factors of level zero representations of $GL_n$}
\author{Rongqing Ye}
\address{Department of Mathematics\\
The Ohio State University}
\email{ye.352@osu.edu}
\date{April 5, 2018}

\begin{document}
\begin{abstract}
For a $p$-adic local field $F$ of characteristic 0, with residue field $\mathfrak{f}$, we prove that the Rankin-Selberg gamma factor of a pair of level zero representations of linear general groups over $F$ is equal to a gamma factor of a pair of corresponding cuspidal representations of linear general groups over $\mathfrak{f}$. Our results can be used to prove a variant of Jacquet's conjecture.
\end{abstract}
\maketitle

\section{Introduction}\label{sec:introduction}

Let $F$ be a $p$-adic local field of characteristic 0, with ring of integers $\mathfrak{o}$, maximal ideal $\mathfrak{p}$ and residue field $\mathfrak{f} = \mathfrak{o}/\mathfrak{p}$ of $q$ elements. Throughout the whole paper, we fix an additive character $\psi$ on $F$ of conductor $\mathfrak{p}$. That is to say, $\psi$ is trivial on $\mathfrak{p}$ but not trivial on $\mathfrak{o}$. Thus, $\psi$ can be descended to a nontrivial character of $\mathfrak{f}$, which will be denoted by $\psi$ as well.

For a ring $R$, let $G_n(R) = GL_n(R)$ be the general linear group with coefficients in $R$. From an irreducible cuspidal representation $\sigma$ of $G_n(\mathfrak{f})$ which can be viewed as a representation of $G_n(\mathfrak{o})$, Bushnell and Kutzko constructed, via their type theory \cite{BushnellKutzko93}, irreducible supercuspidal representations of $G_n(F)$. Among these supercuspidal representations, there is a class of \emph{level zero} representations in the form
$$\pi = \mathrm{ind}_{F^\times G_n(\mathfrak{o})}^{G_n(F)} \Lambda,$$
where $\Lambda$ is a representation of $F^\times G_n(\mathfrak{o})$ restricting to $\sigma$ on $G_n(\mathfrak{o})$. Let $\pi^\prime$ be a level zero representation of $G_m(F)$ coming from an irreducible cuspidal representation $\sigma^\prime$ of $G_m(\mathfrak{f})$. For the pair $(\pi, \pi^\prime)$, we can follow Jacquet, Piatetski-Shapiro and Shalika \cite{JacquetPiShShalika83} to define an important invariant $\gamma(s, \pi \times \pi^\prime, \psi)$. Correspondingly, we consider $\gamma(\sigma \times \sigma^\prime, \psi)$ as in Roditty's master thesis \cite{Roditty10}, Nien \cite{Nien14}, or in Piatetski-Shapiro's unpublished lecture notes \cite{PiatetskiShapiro1976}. In this paper, we are going to express these two gamma factors in terms of Bessel functions to show that these two gamma factors are equal.

\begin{theorem}\label{thm:main}
Let $\pi$ and $\pi^\prime$ be level zero representations of $G_n(F)$ and $G_m(F)$, coming from $\sigma$ and $\sigma^\prime$ respectively. If $n = m$, we assume that $\sigma^\prime$ is not isomorphic to $\tilde{\sigma}$, the contragradient of $\sigma$. Then
\begin{equation*}
\gamma(s, \pi \times \pi^\prime, \psi) = \gamma(\sigma \times \sigma^\prime, \psi).
\end{equation*}
\end{theorem}

Since definitions of gamma factors for $G_n(F) \times G_m(F)$ with $n < m$ and those for $G_n(F) \times G_n(F)$ are different, we split \Cref{thm:main} into two cases, \Cref{thm:GnXGm} and \Cref{thm:GnXGn}, and prove them separately. An implicit consequence of the theorem is that gamma factors for pairs of level zero representations under the assumption in \Cref{thm:main} are constants. This consequence appears also in the work of Bushnell, Henniart and Kutzko \cite{BushnellHenniartKutzko98}, where they give an explicit formula for conductors in terms of quantities from simple strata \cite{BushnellKutzko93}.

Paskunas and Stevens in their work \cite{PaskunasStevens08} constructed explicit Whittaker functions and used them to derive inductive formulas for epsilon factors of pairs of supercuspidal representations of $G_n(F)$. In their formula, epsilon factor of a pair of supercuspidal representations reduces to an epsilon factor of a pair of level zero representations. In a similar fashion, Kim \cite{Kim14} proved inductive formulas for $G_n(F) \times G_m(F)$ where $n > m$. To some extent, our results are some efforts to make these inductive formulas more explicit. 

Moreover, our results imply a variant of Jacquet's conjecture on the local converse theorem. Jacquet's conjecture basically says that an irreducible generic representation $\pi$ of $G_n(F)$ is uniquely determined up to isomorphisms by the family of $\gamma(s, \pi \times \tau, \psi)$, where $\tau$ runs over irreducible generic representations of $G_r(F)$ with $1 \le r \le \left[\frac{n}{2}\right]$. To be more precise, we have

\begin{conjecture} Let $\pi_1$ and $\pi_2$ be irreducible generic representations of $G_n(F)$ sharing the same central character. If
$$\gamma(s, \pi_1 \times \tau, \psi) = \gamma(s, \pi_2 \times \tau, \psi)$$
for all irreducible generic representations of $G_r(F)$ with $1 \le r \le \left[\frac{n}{2}\right]$, then $\pi_1 \cong \pi_2$.
\end{conjecture}

This conjecture has been recently confirmed positively independently by Chai \cite{Chai16} and Jacquet-Liu \cite{JacquetLiu16}. The finite-field analogue of this conjecture was proved by Nien \cite{Nien14}. In this paper, we will make use of \Cref{thm:main} and a slightly improved \Cref{thm:local-converse-finite} on converse problems over a finite field (cf. \cite[Theorem 3.9]{Nien14}), to prove

\begin{reptheorem}{thm:local-converse-level-zero}
Let $\pi_1$ and $\pi_2$ be level zero representations of $G_n(F)$ with the same central character. If
$$\gamma(s, \pi_1 \times \tau, \psi) = \gamma(s, \pi_2 \times \tau, \psi)$$
for all level zero representations of $G_r(F)$ with $1 \le r \le \left[\frac{n}{2}\right]$, then $\pi_1 \cong \pi_2$.
\end{reptheorem}

\section{Notations and Preliminaries}\label{sec:notations}

In this section, we are going to introduce the essential notations and preliminaries for our results. If $R$ is a ring, then we use $M_{r,c}(R)$ to denote the set of $r \times c$ matrices with coefficients in $R$. $U_n(R)$ is the set of standard unipotent matrices with coefficients in $R$:
\begin{equation*}
U_n(R) = \left\{\begin{pmatrix}
1 & * & \cdots & * \\
0 & 1 & \cdots & * \\
\vdots & \vdots & \ddots & \vdots \\
0 & 0 & \cdots & 1
\end{pmatrix}\right\}.
\end{equation*}
If $\psi_R: R \to \C^\times$ is an additive character on $R$, then it can be extended to a character $\psi_R$ of $U_n(R)$, defined by
$$\psi_R((u_{ij})) = \psi_R(u_{12}+u_{23}+\cdots+u_{n-1,n}),$$
for $(u_{ij}) \in U_n(R)$.

\subsection{Level zero representations}\label{sec:level-zero} A representation $\pi$ of $G_n(F)$ is of \emph{level zero} if there exists an irreducible cuspidal representation $\sigma$ of $G_n(\mathfrak{f})$ such that
$$\pi \cong \mathrm{ind}_{F^\times G_n(\mathfrak{o})}^{G_n(F)} \Lambda,$$
where $\Lambda$ is a representation of $F^\times G_n(\mathfrak{o})$ such that $\Lambda|_{G_n(\mathfrak{o})}$ is an inflation of $\sigma$ via $G_n(\mathfrak{o}) \xrightarrow{\mod{\mathfrak{p}}} G_n(\mathfrak{f})$. Here, $\mathrm{ind}$ is smooth compact induction. By theory of types \cite{BushnellKutzko93}, $\pi$ is irreducible supercuspidal.

Let $\omega_\sigma$ be the central character of $\sigma$ and $\omega_\Lambda$ be the central character of $\Lambda$. Since $\Lambda|_{G_n(\mathfrak{o})} = \sigma$, $\omega_\Lambda|_{\mathfrak{o}^\times} = \omega_\sigma|_{\mathfrak{o}^\times}$. Thus, given $\sigma$, $\omega_\Lambda$ is uniquely determined by the complex number $\lambda = \omega_\Lambda(\varpi) \in \C^\times$, where $\varpi$ is a uniformizer of $F$. From a level zero representation $\pi$, we can the get a pair $(\lambda, \sigma)$. Conversely, because $\Lambda$ is determined by $\omega_\Lambda$ and $\sigma$, such a pair $(\lambda, \sigma)$ defines a level zero representation. By \cite[Theorem 8.4.1]{BushnellKutzko93}, the correspondence between the set of level zero representations of $G_n(F)$ and the set of pairs $(\lambda, \sigma)$ where $\lambda \in \C^\times$ and $\sigma$ is an irreducible cuspidal representation of $G_n(\mathfrak{f})$ is a bijection. Henceforth, by a level zero representation $\pi$ coming from $(\lambda, \sigma)$, we mean that $\pi$ corresponds to $(\lambda, \sigma)$ in this bijection. We sometime will just say a level zero representation $\pi$ is from $\sigma$, when the missing $\lambda$ won't cause any confusion.

Let $|\cdot|$ be the absolute value on $F$ such that $|\varpi|=q^{-1}$. For any complex number $s$, $|\cdot|^s$ is a one dimensional representation of $G_n(F)$ given by $|\det(g)|^s$.

\begin{lemma}\label{lem:level-zero-repn}
Let $\pi$ be a level zero representation of $G_n(F)$ coming from $(\lambda, \sigma)$, where $\lambda \in \C^\times$ and $\sigma$ is an irreducible cuspidal representation of $G_n(\mathfrak{f})$. Then,
\begin{enumerate}[(1)]
\item The contragradient $\tilde{\pi}$ is a level zero representation from $(\lambda^{-1}, \tilde{\sigma})$.
\item For any complex number $s$, $\pi \otimes |\cdot|^s$ is a level zero representation from $(\lambda q^{-ns}, \sigma)$.
\end{enumerate}
\end{lemma}

\begin{proof}
These are consequences of the facts
$$\left(\mathrm{ind}_{F^\times G_n(\mathfrak{o})}^{G_n(F)} \Lambda\right)^{\tilde{}} \cong \mathrm{ind}_{F^\times G_n(\mathfrak{o})}^{G_n(F)} \tilde{\Lambda},$$
and
$$\left(\mathrm{ind}_{F^\times G_n(\mathfrak{o})}^{G_n(F)} \Lambda\right) \otimes |\cdot|^s \cong \mathrm{ind}_{F^\times G_n(\mathfrak{o})}^{G_n(F)} \left(\Lambda \otimes |\cdot|^s\right). $$
\end{proof}

\begin{corollary}\label{cor:l-function}
Let $\pi_i$ be level zero representations of $G_n(F)$ coming from $(\lambda_i, \sigma_i)$, where $\lambda_i \in \C^\times$ and $\sigma_i$ is an irreducible cuspidal representation of $G_n(\mathfrak{f})$, for $i = 1, 2$. If $\sigma_1 \cong \tilde{\sigma}_2$, then
$$L(s, \pi_1 \times \pi_2) = (1-\lambda_1\lambda_2q^{ns})^{-1}.$$
\end{corollary}

\begin{proof}
$L(s, \pi_1 \times \pi_2)$ has a pole at $s_0$ if and only if $\pi_1 \otimes |\cdot|^{s_0} \cong \tilde{\pi}_2$. By \Cref{lem:level-zero-repn}, $\pi_1 \otimes |\cdot|^{s_0} \cong \tilde{\pi}_2$ is equivalent to $\lambda_1\lambda_2 = q^{ns_0}$. There will be $n$ poles, and if we fix one such pole $s_0$, these $n$ poles are
$$s_0, s_0 e^{\frac{2\pi i}{n\ln q}}, \dots, s_0 e^{\frac{2\pi i (n-1)}{n\ln q}}.$$
Therefore,
\begin{equation*}
\begin{split}
L(s, \pi_1 \times \pi_2) &= \prod_{k = 0}^{n-1} \left(1 - q^{s_0 - s}e^{\frac{2\pi ik}{n}}\right)^{-1}
= (1 - q^{ns_0-ns})^{-1} \\
&= (1 - \lambda_1\lambda_2q^{ns})^{-1}.
\end{split}
\end{equation*}
\end{proof}

\subsection{Gamma factors}\label{sec:gamma}

For simplicity, we consider only supercuspidal representations and content ourselves with minimal definitions and facts. We have to point out in advance that our choices of Haar measures are different from the ``standard'' ones. We choose the Haar measure $dx$ on $F$ so that $\int_{\mathfrak{p}} dx = 1$, and choose the Haar measure $d^\times x$ on $F^\times$ so that $1+\mathfrak{p}$ has volume $1$. We then normalize all other Haar measures accordingly.

We follow \cite{JacquetPiShShalika83} to define Rankin-Selberg gamma factors over $F$. Let $\pi$ and $\pi^\prime$ be irreducible supercuspidal representations of $G_n(F)$ and $G_m(F)$ respectively. Let $\mathcal{W}(\pi, \psi)$ be the Whittaker model of $\pi$ with respect to $\psi$. Similarly, we have the Whittaker model $\mathcal{W}(\pi^\prime, \psi^{-1})$ for $\pi^\prime$ with respect to $\psi^{-1}$.

For the case $n > m$, one can take any $W \in \mathcal{W}(\pi, \psi)$ and $W^\prime \in \mathcal{W}(\pi^\prime, \psi^{-1})$ to form the local Rankin-Selberg integrals:

\begin{equation*}
\begin{split}
\Psi(s; W, W^\prime) &= \int W\left(\begin{array}{cc}h & \\ & I_{n-m}\end{array}\right)W^\prime(h)|\det h|^{s - (n-m)/2} \, dh \\
\widetilde{\Psi}(s; W, W^\prime) &= \iint W\left(\begin{array}{ccc}h & & \\ x & I_{n-m-1} & \\ & & 1\end{array}\right)W^\prime(h) |\det h|^{s - (n-m)/2} \, dxdh
\end{split}
\end{equation*}
In the above integrations, $h$ is over $U_m(F) \backslash G_m(F)$ and $x$ is over $M_{n-m-1, m}(F)$. These integrals converge for $\mathrm{Re}\, s \gg 0$, have analytic continuations to $\C$ and satisfy a functional equation, see \cite{JacquetPiShShalika83}. The gamma factor $\gamma(s, \pi \times \pi^\prime, \psi)$ connects two integrals in the functional equation
\begin{equation}
\widetilde{\Psi}(1-s; \rho(w_{n, m})\widetilde{W}, \widetilde{W^\prime}) = \gamma(s, \pi \times \pi^\prime, \psi) \Psi(s; W, W^\prime),
\label{eqn:local-gamma-GnXGm}
\end{equation}
where
\begin{equation*}
w_{n, m} = \begin{pmatrix}
I_m & \\
& w_{n-m}
\end{pmatrix}
\text{ with }
w_{n-m} = \begin{pmatrix}
& & 1 \\
& \iddots & \\
1 & &
\end{pmatrix} \in G_{n-m}(F).
\end{equation*}

For the case $n=m$, one must introduce another parameter $\phi \in \mathcal{S}(F^n)$, which is the space of Schwartz functions on $F^n$. The gamma factor in this case comes from integrals
\begin{equation*}
\Psi(s; W, W^\prime, \phi) = \int_{U_n(F) \backslash G_n(F)} W(g)W^\prime(g)\phi(e_ng)|\det g|^s \, dg,
\end{equation*}
where $e_n = (0, \dots, 0, 1) \in F^n$. From \cite{JacquetPiShShalika83} again, these integrals converge for $\mathrm{Re}\, s \gg 0$, can be analytically continued to $\C$, and satisfy a function equation. The corresponding functional equation gives us the desired gamma factor:
\begin{equation}
\Psi(1-s; \widetilde{W}, \widetilde{W^\prime}, \widehat{\phi}) = \gamma(s, \pi \times \pi^\prime, \psi)\Psi(s; W, W^\prime, \phi),
\label{eqn:local-gamma-GnXGn}
\end{equation}
where
\begin{equation*}
\widehat{\phi}(x) = \int_{F^n} \phi(y)\psi(<x, y>) \, dy.
\end{equation*}
Here, $<x, y>$ is the usual inner product on the vector space $F^n$, and $dy$ is the Haar measure on $F^n$ coming from the Haar measure $dx$ on $F$ such that $\int_\mathfrak{p} \, dx = 1$. 

Gamma factors for finite fields are defined in a very similar way, only with integrals replaced by finite sums. An irreducible representation $\sigma$ of $G_n(\mathfrak{f})$ is \emph{cuspidal} if it has no $U_n(\mathfrak{f})$-fixed vectors, and it has Whittaker model, see \cite{Gelfand70}. Let $\sigma$ be an irreducible cuspidal representation of $G_n(\mathfrak{f})$ with Whittaker model $\mathcal{W}(\sigma, \psi)$, and let $\sigma^\prime$ be an irreducible cuspidal of $G_m(\mathfrak{f})$ with Whittaker model $\mathcal{W}(\sigma^\prime, \psi^{-1})$.

In the case where $n > m$, we refer readers to Piatetski-Shapiro's unpublished notes \cite{PiatetskiShapiro1976}, Theorem 5.1 and 5.4 of \cite{Roditty10} and Theorem 2.10 of \cite{Nien14} for details. The factor $\gamma(\sigma \times \sigma^\prime, \psi)$ is defined in the equation
\begin{equation}
\begin{split}
\gamma(\sigma \times \sigma^\prime, \psi)&\sum_{h \in U_m(\mathfrak{f}) \backslash G_m(\mathfrak{f})}
W\begin{pmatrix}h & \\ & I\end{pmatrix} W^\prime(h) \\
&= \sum_{\substack{h \in U_m(\mathfrak{f}) \backslash G_m(\mathfrak{f}) \\ x \in M_{m,n-m-1}(\mathfrak{f})}} W\begin{pmatrix}& 1 & \\ & & I \\ h & & x\end{pmatrix} W^\prime(h).
\end{split}
\label{eqn:finite-gamma-GnXGm}
\end{equation}
for any $W \in \mathcal{W}(\sigma, \psi)$ and $W^\prime \in \mathcal{W}(\sigma^\prime, \psi^{-1})$.

In the case where $n = m$, taking into account of another parameter $\phi \in C_0(\mathfrak{f}^n)$, the set of $\C$-valued functions on $\mathfrak{f}^n$ valued $0$ at $0 \in \mathfrak{f}^n$, we have
\begin{equation}
\begin{split}
\gamma(\sigma \times \sigma^\prime, \psi) & \sum_{g \in U_n(\mathfrak{f}) \backslash G_n(\mathfrak{f})} W(g)W^\prime(g)\phi(e_n g) \\
&= \sum_{g \in U_n(\mathfrak{f}) \backslash G_n(\mathfrak{f})}W(g)W^\prime(g)\widehat{\phi}(e_1\, ^tg^{-1}),
\end{split}
\label{eqn:finite-gamma-GnXGn}
\end{equation}
where $e_1 = (1, 0, \dots, 0)$, and
\begin{equation*}
\widehat{\phi}(x) = \sum_{y \in \mathfrak{f}^n} \phi(y)\psi(<x, y>).
\end{equation*}
Piateski-Shapiro showed in his lecture \cite{PiatetskiShapiro1976} that $\gamma(\sigma \times \sigma^\prime, \psi)$ in \Cref{eqn:finite-gamma-GnXGn} is well-defined. Since his notes were never published, we shall include a proof of a special case that we need.

\begin{theorem}\label{thm:def-gamma-GnXGn}
Let $\sigma$ and $\tau$ be an irreducible cuspidal representations of $G_n(\mathfrak{f})$. There exist a unique complex number $\gamma(\sigma \times \tau, \psi)$ such that \Cref{eqn:finite-gamma-GnXGn} holds for any $W \in \mathcal{W}(\sigma, \psi)$, $W^\prime \in \mathcal{W}(\sigma^\prime, \psi^{-1})$ and any function $\phi$ on $\mathfrak{f}^n$.
\end{theorem}

\begin{proof}
Let $C_0(\mathfrak{f}^n)$ be the set of $\C$-valued functions $\phi$ on $\mathfrak{f}^n$ such that $\phi(0)=0$. $G_n(\mathfrak{f})$ can acts on $C_0(\mathfrak{f}^n)$ by right multiplication:
$$(R(g)\phi)(x) = \phi(xg),$$
for $g \in G_n(\mathfrak{f})$ and $\phi \in C_0(\mathfrak{f}^n)$. If we set
$$L_1(W, W^\prime, \phi) = \sum_{g \in U_n(\mathfrak{f}) \backslash G_n(\mathfrak{f})} W(g)W^\prime(g)\phi(e_n g),$$
and
$$L_2(W, W^\prime, \phi) = \sum_{g \in U_n(\mathfrak{f}) \backslash G_n(\mathfrak{f})}W(g)W^\prime(g)\widehat{\phi}(e_1\, ^tg^{-1}),$$
we can check $L_1$ and $L_2$ are $G_n(\mathfrak{f})$-invariant trilinear forms on $\sigma \otimes \tau \otimes C_0(\mathfrak{f}^n)$. It is then enough to show that such forms are unique up to scalars. Let $(v, w, \phi)$ be a $G_n(\mathfrak{f})$-invariant trilinear form on $\sigma \otimes \tau \otimes C_0(\mathfrak{f}^n)$. Let $\phi_n$ be the indication function of $e_n$, and consider a bilinear form
$$(v, w)_{P_n} = (v, w, \phi_n)$$
on $\sigma \otimes \tau$. Since $\phi_n$ is fixed by $P_n(\mathfrak{f})$, the group of matrices with last row being $(0, \dots, 0, 1)$, $(v, w)_{P_n}$ is a $P_n(\mathfrak{f})$-invariant bilinear form on $\sigma \otimes \tau$. Conversely, a $P_n(\mathfrak{f})$-invariant bilinear form determines uniquely a $G_n(\mathfrak{f})$-invariant trilinear form: for any $\phi \in C_0(\mathfrak{f}^n)$, since $G_n(\mathfrak{f})$ acts transitively $\mathfrak{f}^n \backslash \{0\}$, we can write
$$\phi = \sum_{g \in G_n(\mathfrak{f}) / P_n(\mathfrak{f})}c_g R(g)\phi_n,$$
for some $c_g \in \C$. We can then define $(v, w, \phi)$ on $\sigma \otimes \tau \otimes C_0(\mathfrak{f}^n)$ as
$$(v, w, \phi) = \sum_{g \in G_n(\mathfrak{f}) / P_n(\mathfrak{f})} c_g \cdot (g^{-1}v, g^{-1}w)_{P_n}.$$
$(v, w, \phi)$ can be easily checked to be $G_n(\mathfrak{f})$-invariant. Therefore, it suffices to prove that such $P_n(\mathfrak{f})$-invariant bilinear forms are unique up to constants.

Since $\sigma$ and $\tau$ are cuspidal, $\sigma|_{P_n(\mathfrak{f})} = \mathrm{Ind}_{U_n(\mathfrak{f})}^{P_n(\mathfrak{f})} \psi$ and $\tau|_{P_n(\mathfrak{f})} = \mathrm{Ind}_{U_n(\mathfrak{f})}^{P_n(\mathfrak{f})} \psi^{-1}$ are irreducible, see for example \cite[Theorem 2.3]{Gelfand70}. Since a $P_n(\mathfrak{f})$-invariant bilinear form identifies $\sigma$ with the contragradient of $\tau$, by Schur's lemma,
\begin{equation*}
\dim \mathrm{Hom}_{P_n(\mathfrak{f})}(\sigma \otimes \tau, \C) \le 1,
\end{equation*}
viewing $\C$ as a trivial representation of $P_n(\mathfrak{f})$.
\end{proof}

\subsection{Paskunas-Stevens Whittaker functions}\label{sec:explicit-whittaker}

Suppose from now on that $\pi$ is a level zero representation of $G_n(F)$ coming from an irreducible cuspidal representation $\sigma$ of $G_n(\mathfrak{f})$, as defined in \Cref{sec:introduction} (see \cite{BushnellKutzko93,PaskunasStevens08} for more details). We write $K_n = G_n(\mathfrak{o})$, and denote $\overline{k}$ the image of $k$ under the natural map $K_n = G_n(\mathfrak{o}) \xrightarrow{\mod{\mathfrak{p}}} G_n(\mathfrak{f})$. By Theorem 5.8 of \cite{PaskunasStevens08}, there is a Whittaker function $W_\pi$ of $\pi$ with respect to $\psi$ such that
\begin{enumerate}[(1)]
\item\label{p1} $\mathrm{Supp} W_{\pi} \subseteq U_n(F)F^\times K_n$, and 
\begin{equation}
W_{\pi}(ug) = \psi(u) \mathcal{J}_{\pi}(g) \text{ for any } u \in U_n(F), g \in F^\times K_n,
\label{eqn:explicit-whittaker}
\end{equation}
where for $g=ak \in F^\times K_n$ with $a \in F^\times$ and $k \in K_n$,
\begin{equation}
\mathcal{J}_{\pi}(g) = \omega_{\pi}(a)|U_n(\mathfrak{f})|^{-1}\sum_{h \in U_n(\mathfrak{f})} \psi(h^{-1})\mathrm{tr}_\sigma(\overline{k}h),
\label{eqn:local-bessel}
\end{equation}
where $\omega_\pi$ is the central character of $\pi$, and $\mathrm{tr}_{\sigma}$ is the trace of $\sigma$. Note that by definition of level zero representations, $\mathcal{J}_\pi$ and $\omega_\pi$ agree on $F^\times \cap K_n$.

\item\label{p2} If $P_n$ is the mirabolic subgroup of $G_n(F)$, i.e. the group of matrices with last row being $(0, \dots, 0, 1)$, then $\mathrm{Supp} W_\pi \cap P_n = U_n(F)(H_n^1 \cap P_n)$, and for $u \in U_n(F)$ and $h \in H_n^1 \cap P_n$,
$$W_\pi(uh) = \psi(u).$$
Here, $H_n^1$ is the subgroup of $K_n$ consisting of matrices that will reduce to identity modulo $\mathfrak{p}$. 
\end{enumerate}

Since $\sigma$ is an irreducible cuspidal representation of $G_n(\mathfrak{f})$, by \cite{Gelfand70}, it has a unique Bessel function $J_{\sigma, \psi}$ with respect to $\psi$ such that $J_{\sigma, \psi}(I)=1$. Moreover, by Proposition 4.5 of \cite{Gelfand70}, for $g \in G_n(\mathfrak{f})$
\begin{equation}
J_{\sigma, \psi}(g) = |U_n(\mathfrak{f})|^{-1}\sum_{h \in U_n(\mathfrak{f})} \psi(h^{-1})\mathrm{tr}_\sigma(gh).
\label{eqn:finite-bessel}
\end{equation}
Therefore, from \Cref{eqn:local-bessel,eqn:finite-bessel}, we have
\begin{lemma}\label{lem:bessels-relation} For $a \in F^\times$ and $k \in K_n$,
\begin{equation*}
\mathcal{J}_\pi(ak) = \omega_\pi(a) J_{\sigma, \psi}(\overline{k}).
\end{equation*}
\end{lemma}

\section{$G_n \times G_m$ case}\label{sec:GnXGm}

In this section, we are going to investigate gamma factors for $G_n \times G_m$, where $n > m$. 

\begin{theorem}\label{thm:GnXGm}
Let $\pi$ and $\pi^\prime$ be depth zero representations coming from $\sigma$ and $\sigma^\prime$ respectively. Then
\begin{equation*}
\gamma(s, \pi \times \pi^\prime, \psi) = \gamma(\sigma \times \sigma^\prime, \psi).
\end{equation*}
\end{theorem}

\begin{proof}
Let $W_{\pi}$ be the Paskunas-Stevens Whittaker function of $\pi$ with respect to $\psi$, and let $W_{\pi^\prime}$ be the Paskunas-Stevens Whittaker function of $\pi^\prime$ with respect to $\psi^{-1}$, as in \Cref{sec:explicit-whittaker}. Then
\begin{equation*}
\Psi(s; W_{\pi}, W_{\pi^\prime}) = \int_{U_m(F) \backslash G_m(F)} W_{\pi}\left(\begin{array}{cc}h & \\ & I_{n-m}\end{array}\right)W_{\pi^\prime}(h)|\det h|^{s - (n-m)/2} \, dh.
\end{equation*}
By property ($\ref{p2}$) of $W_{\pi}$, in order to support $W_{\pi}$, we must have
\begin{equation*}
\begin{pmatrix}
h & \\
& I_{n-m}
\end{pmatrix} \in U_n(F)(H^1_n \cap P_n).
\end{equation*}
Then $h$ can be written as $ux$, with $u \in U_m(F)$ and $x \in H_m^1$ and thus $\det h \in 1+\mathfrak{p}$. From here, we can deduce that $|\det h| = 1$, and property ($\ref{p2}$) says
$$W_{\pi}\left(\begin{array}{cc}h & \\ & I_{n-m}\end{array}\right) = \psi(u).$$
Since $x \in H_m^1 \subset K_m$, $\overline{x} = I$, so by \Cref{lem:bessels-relation},
\begin{equation*}
W_{\pi^\prime}(ux) = \psi^{-1}(u) \mathcal{J}_{\pi^\prime}(x) = \psi^{-1}(u)J_{\sigma^\prime, \psi^{-1}}(I) = \psi^{-1}(u).
\end{equation*}
Therefore, under our choices of Haar measures,
\begin{equation}
\Psi(s; W_{\pi}, W_{\pi^\prime}) = \int_{U_m(F) \cap H^1_m \backslash H^1_m} \, dh = 1.
\label{eqn:Psi-GnXGm}
\end{equation}

If we unfold $\widetilde{\Psi}(1-s; \rho(w_{n, m})\widetilde{W_{\pi}}, \widetilde{W_{\pi^\prime}})$ from its definition, then the following term from $\rho(w_{n, m})\widetilde{W_{\pi}}$ appears in the integral:
\begin{equation*}
W_{\pi}
\begin{pmatrix}
& 1 & \\
& & I \\
w_m\, ^th^{-1} & & -w_m\, ^th^{-1} \, ^tx\, w_{n-m-1}
\end{pmatrix}
\end{equation*}
Using changes of variables, we have for $\mathrm{Re}\, s \ll 0$,
\begin{equation*}
\begin{split}
& \widetilde{\Psi}(1-s; \rho(w_{n, m})\widetilde{W_{\pi}}, \widetilde{W_{\pi^\prime}}) \\
&= \iint W_{\pi}\begin{pmatrix}& 1 & \\ & & I \\ h & & x\end{pmatrix} W_{\pi^\prime}(h)|\det h|^{s-2+(n-m)/2} \, dxdh,
\end{split}
\end{equation*}
where $h$ is integrated over $U_m(F) \backslash G_m(F)$ and $x$ is integrated over $M_{m, n-m-1}(F)$. By property ($\ref{p1}$) of $W_\pi$, there exist $u = (u_{ij}) \in U_n(F)$, $t \in \N$ and $k = (k_{ij}) \in K_n$ so that
\begin{equation*}
\begin{pmatrix}& 1 & \\ & & I \\ h & & x\end{pmatrix} = \varpi^t uk.
\end{equation*}
Comparing the (m+1)-th column, we get
$$k_{n, m+1} = k_{n-1, m+1} = \cdots = k_{2, m+1} = 0, \text{ and } k_{1, m+1}\varpi^t = 1.$$
Since $k \in K_n$, $\det k \in \mathfrak{o}^\times$, so $k_{1, m+1}$ must then be in $\mathfrak{o}^\times$, thus $t = 0$ and $k_{1, m+1} = 1$. Taking determinants,
$$\left|\det \begin{pmatrix}& 1 & \\ & & I \\ h & & x\end{pmatrix}\right| = |\det h|, \text{ and } |\det uk| = 1.$$
Therefore, $|\det h| = 1$. In other words, $\widetilde{\Psi}(1-s; \rho(w_{n, m})\widetilde{W_{\pi}}, \widetilde{W_{\pi^\prime}})$ is independent of $s$. Now we will write $u$ and $k$ in terms of block matrices.
\begin{equation*}
u = \begin{pmatrix} u_1 & v \\ & u_2\end{pmatrix}, 
k = \begin{pmatrix} k_{11} & k_{12} \\ k_{21} & k_{22}\end{pmatrix},
\end{equation*}
where $u_1$ is of size $(n-m) \times (n-m)$ and $k_{11}$ is of size $(n-m) \times m$. Then
\begin{equation*}
\begin{pmatrix} & I \\ h & x^\prime\end{pmatrix} = uk, \text{ where } x^\prime = \begin{pmatrix}0 & x\end{pmatrix}
\end{equation*}
gives
\begin{equation*}
\begin{split}
u_1k_{11} + vk_{21} &= 0 \\
u_1k_{12} + vk_{22} &= I \\
u_2k_{21} &= h \\
u_2k_{22} &= x^\prime
\end{split}
\end{equation*}
Since $h$ is invertible, from the third equality, $k_{21}$ is invertible. From the first two equalities above, we can deduce that
\begin{equation*}
\begin{split}
v &= -u_1k_{11}k_{21}^{-1} \\
u_1 &= (k_{12} - k_{11}k_{21}^{-1}k_{22})^{-1}
\end{split}
\end{equation*}
The third equality tells us that $\det k_{21} = \det h \in \mathfrak{o}^\times$. Hence, $k_{21} \in K_m$, from which we deduces that both $u_1$ and $v$ have all coefficients in $\mathfrak{o}$. Therefore, we might assume that $u_1 = I$ and $v=0$, which then implies $k_{11} = 0$ and $k_{12} = I$. Since in the integral, $h$ is integrated over $U_m(F) \backslash G_m(F)$, by the third equality and that $k_{21} \in K_m$, $h$ can be integrated only over $U_m(F) \cap K_m \backslash K_m$. And if $h \in K_m$, then $u_2 \in K_m$, the forth equality then implies $x$ has all coefficients in $\mathfrak{o}$. Therefore,
\begin{equation}
\begin{split}
& \widetilde{\Psi}(1-s; \rho(w_{n, m})\widetilde{W_{\pi}}, \widetilde{W_{\pi^\prime}}) \\
&= \int_{U_m(F) \cap K_m \backslash K_m} \int_{M_{m,n-m-1}(\mathfrak{o})}W_{\pi}\begin{pmatrix}& 1 & \\ & & I \\ h & & x\end{pmatrix} W_{\pi^\prime}(h) \, dxdh \\
&= q^{-m(m-1)/2}\int_{K_m} \int_{M_{m, n-m-1}(\mathfrak{o})}W_{\pi}\begin{pmatrix}& 1 & \\ & & I \\ h & & x\end{pmatrix} W_{\pi^\prime}(h) \, dxdh \\
&= q^{-m(m-1)/2} \sum_{\substack{h \in G_m(\mathfrak{f}) \\ x \in M_{m,n-m-1}(\mathfrak{f})}} J_{\sigma, \psi}\begin{pmatrix}& 1 & \\ & & I \\ h & & x\end{pmatrix} J_{\sigma^\prime, \psi^{-1}}(h)
\end{split}
\label{eqn:tildePsi-GnXGm}
\end{equation}
The second equality in \Cref{eqn:tildePsi-GnXGm} comes from the fact that under our choices of Haar measures,
\begin{equation*}
\int_{U_m(F) \cap K_m} \, dh = q^{m(m-1)/2}.
\end{equation*}
The third equality is a consequence of Property (\ref{p1}) of Paskunas-Stevens Whittaker functions and \Cref{lem:bessels-relation}. \Cref{eqn:tildePsi-GnXGm}, together with \Cref{eqn:Psi-GnXGm} allows us to express $\gamma(s, \pi \times \pi^\prime, \psi)$ in terms of Bessel functions:
\begin{equation}
\begin{split}
\gamma(s, \pi \times \pi^\prime, \psi) &= q^{-m(m-1)/2} \sum_{\substack{h \in G_m(\mathfrak{f}) \\ x \in M_{m,n-m-1}(\mathfrak{f})}} J_{\sigma, \psi}\begin{pmatrix}& 1 & \\ & & I \\ h & & x\end{pmatrix} J_{\sigma^\prime, \psi^{-1}}(h) \\
&= \sum_{\substack{h \in U_m(\mathfrak{f}) \backslash G_m(\mathfrak{f}) \\ x \in M_{m,n-m-1}(\mathfrak{f})}} J_{\sigma, \psi}\begin{pmatrix}& 1 & \\ & & I \\ h & & x\end{pmatrix} J_{\sigma^\prime, \psi^{-1}}(h).
\end{split}
\label{eqn:local-gamma-in-bessel-GnXGm}
\end{equation}

Next, we need to write $\gamma(\sigma \times \sigma^\prime, \psi)$ in terms of Bessel functions. To this end, we put $J_{\sigma, \psi} \in \mathcal{W}(\sigma, \psi)$ and $J_{\sigma^\prime, \psi^{-1}} \in \mathcal{W}(\sigma^\prime, \psi)$ in \Cref{eqn:finite-gamma-GnXGm} to calculate $\gamma(\sigma \times \sigma^\prime, \psi)$. Since by \cite{Gelfand70}, the support for $J_{\sigma, \psi}$ are matrices of the form
$$u_1\begin{pmatrix}
 & & & \delta_1 I_{n_1} \\ 
 & & \delta_2 I_{n_2} & \\
 & \iddots & & \\
 \delta_r I_{n_r} & & &
\end{pmatrix}u_2,$$
where $u_1, u_2 \in U_n(\mathfrak{f}), \delta_1, \dots, \delta_r \in \mathfrak{f}$ and $n_1 + \cdots + n_r = n$. In particular, if $J_{\sigma, \psi}\begin{pmatrix}h & \\ & I\end{pmatrix} \ne 0$, then $u \in U_n(\mathfrak{f})$. Thus,
$$\sum_{h \in U_m(\mathfrak{f}) \backslash G_m(\mathfrak{f})}J_{\sigma, \psi}\begin{pmatrix}h & \\ & I\end{pmatrix} J_{\sigma^\prime, \psi^{-1}}(h) = 1.$$
Therefore, from \Cref{eqn:finite-gamma-GnXGm},
\begin{equation}
\gamma(\sigma \times \sigma^\prime, \psi) = \sum_{\substack{h \in U_m(\mathfrak{f}) \backslash G_m(\mathfrak{f}) \\ x \in M_{m,n-m-1}(\mathfrak{f})}} J_{\sigma, \psi}\begin{pmatrix}& 1 & \\ & & I \\ h & & x\end{pmatrix} J_{\sigma^\prime, \psi^{-1}}(h).
\label{eqn:finite-gamma-in-bessel-GnXGm}
\end{equation}
One can find a more detailed proof of a simplified version of \Cref{eqn:finite-gamma-in-bessel-GnXGm} in \cite[Proposition 2.16]{Nien14}.

Finally, comparing \Cref{eqn:local-gamma-in-bessel-GnXGm,eqn:finite-gamma-in-bessel-GnXGm}, we conclude that
\begin{equation*}
\gamma(s, \pi \times \pi^\prime, \psi) = \gamma(\sigma \times \sigma^\prime, \psi).
\end{equation*}
\end{proof}

\section{$G_n \times G_n$ case}\label{sec:GnXGn}

We are now consider the case where $n = m$. In this case, we get a very similar theorem.
\begin{theorem}\label{thm:GnXGn}
Let $\pi$ and $\pi^\prime$ be level zero representations of $GL_n(F)$ coming from $\sigma$ and $\sigma^\prime$. Assume that $\sigma^\prime$ is not isomorphic to $\tilde{\sigma}$, the contragradient of $\sigma$. Then
\begin{equation*}
\gamma(s, \pi \times \pi^\prime, \psi) = \gamma(\sigma \times \sigma^\prime, \psi).
\end{equation*}
\end{theorem}

\begin{proof}
As before, let $W_\pi$ and $W_{\pi^\prime}$ be Paskunas-Stevens' Whittaker functions of $\pi$ and $\pi^\prime$ respectively. Let $\phi$ to be an indicator function on $e_n H_n^1$. Proposition 7.2 of \cite{PaskunasStevens08} shows that
\begin{equation}
\Psi(s; W_\pi, W_{\pi^\prime}, \phi) = \int_{U_n(F)\backslash U_n(F)H_n^1} \, dg = 1.
\label{eqn:glnxgln-Psi}
\end{equation}
So from \Cref{eqn:local-gamma-GnXGn}, for $\mathrm{Re}\, s \ll 0$,
$$\gamma(s, \pi \times \pi^\prime, \psi) = \Psi(1-s; \widetilde{W_\pi}, \widetilde{W_{\pi^\prime}}, \widehat{\phi}).$$
Unfolding from the definition of $\Psi$ and using a change of variable $g \mapsto w_n \, ^tg^{-1}$, we get
\begin{equation*}
\gamma(s, \pi \times \pi^\prime, \psi) = \int_{U_n(F) \backslash G_n(F)} W_\pi(g) W_{\pi^\prime}(g) \widehat{\phi}(e_1 \, ^tg^{-1}) |\det g|^{s-1} \, dg.
\end{equation*}
Since $\mathrm{Supp} W \subseteq U_n(F)F^\times K_n = \coprod_{l \in \Z} \varpi^l U_n(F)K_n$, where $\varpi$ is a uniformizer of $F$, we can further express $\gamma$ as
\begin{equation*}
\begin{split}
\gamma(s, \pi \times \pi^\prime, \psi) & = \sum_{l \in \Z} q^{-nl(s-1)} \omega(\varpi)^l\\
& \cdot \int_{U_n(\mathfrak{o}) \backslash K_n} \mathcal{J}_{\pi}(g)\mathcal{J}_{\pi^\prime}(g) \widehat{\phi}(\varpi^{-l}e_1 \, ^tg^{-1}) \, dg,
\end{split}
\end{equation*}
for $\mathrm{Re}\, s \ll 0$, where $\omega$ is the product of the central characters of $\pi$ and $\pi^\prime$. Since $\phi$ is an indicator function on $e_nH_n^1$, we can calculate its Fourier transform.
\begin{equation*}
\widehat{\phi}(\varpi^{-l}e_1 \, ^tg^{-1}) = \left\{
\begin{array}{ll}
1, & l < 0; \\
\psi(e_1\, ^tg^{-1} \, ^te_n), & l = 0; \\
0, & l > 0.
\end{array}
\right.
\end{equation*}
Therefore, together with \Cref{lem:bessels-relation}, for $\mathrm{Re}\, s \ll 0$,
\begin{equation}
\begin{split}
\gamma(s, \pi \times \pi^\prime, \psi) &=  \left(\sum_{l < 0} q^{-nl(s-1)} \omega(\varpi)^l\right) \int_{U_n(\mathfrak{o}) \backslash K_n} \mathcal{J}_{\pi}(g)\mathcal{J}_{\pi^\prime}(g)  \, dg\\
& + \int_{U_n(\mathfrak{o}) \backslash K_n} \mathcal{J}_{\pi}(g)\mathcal{J}_{\pi^\prime}(g) \psi(e_1 \, ^tg^{-1}\, ^te_n)  \, dg \\
&=  \left(\sum_{l < 0} q^{-nl(s-1)} \omega(\varpi)^l\right) \sum_{U_n(\mathfrak{f}) \backslash G_n(\mathfrak{f})} J_{\sigma, \psi}(g)J_{\sigma^\prime, \psi^{-1}}(g) \\
& + \sum_{U_n(\mathfrak{f}) \backslash G_n(\mathfrak{f})} J_{\sigma, \psi}(g)J_{\sigma^\prime, \psi^{-1}}(g) \psi(e_1 \, ^tg^{-1}\, ^te_n) .
\end{split}
\label{eqn:cal_of_gamma_GnXGn}
\end{equation}
If we know
\begin{equation}
\sum_{U_n(\mathfrak{f}) \backslash G_n(\mathfrak{f})} J_{\sigma, \psi}(g)J_{\sigma^\prime, \psi^{-1}}(g) = 0,
\label{eqn:sum-of-two-bessels}
\end{equation}
then
\begin{equation}
\gamma(s, \pi \times \pi^\prime, \psi) =\sum_{U_n(\mathfrak{f}) \backslash G_n(\mathfrak{f})} J_{\sigma, \psi}(g)J_{\sigma^\prime, \psi^{-1}}(g) \psi(e_1 \, ^tg^{-1}\, ^te_n) .
\label{eqn:local-gamma-in-bessel-GnXGn}
\end{equation}

On the other hand, if we take $W = J_{\sigma, \psi}$, $W^\prime = J_{\sigma^\prime, \psi^{-1}}$ and $\phi$ to be an indicator function on $e_n$, then from \Cref{eqn:finite-gamma-GnXGn}, we have
\begin{equation}
\gamma(\sigma \times \sigma^\prime, \psi) = \sum_{U_n(\mathfrak{f}) \backslash G_n(\mathfrak{f})} J_{\sigma, \psi}(g)J_{\sigma^\prime, \psi^{-1}}(g) \psi(e_1 \, ^tg^{-1}\, ^te_n) .
\label{eqn:finite-gamma-in-bessel-GnXGn}
\end{equation}
By \Cref{eqn:local-gamma-in-bessel-GnXGn,eqn:finite-gamma-in-bessel-GnXGn}, we conclude that
\begin{equation*}
\gamma(s, \pi \times \pi^\prime, \psi) =\gamma(\sigma \times \sigma^\prime, \psi).
\end{equation*}
Thus, to finish the proof, it remains to prove \Cref{eqn:sum-of-two-bessels}. This follows from the following lemma.
\end{proof}

\begin{lemma}\label{lem:representations-and-bessels}
Let $\sigma$ and $\sigma^\prime$ be two irreducible generic representations of $G_n(\mathfrak{f})$ with normalized Bessel functions $J_{\sigma, \psi}$ and $J_{\sigma^\prime, \psi^{-1}}$. Then
\begin{equation*}
\sum_{g \in U_n(\mathfrak{f}) \backslash G_n(\mathfrak{f})} J_{\sigma, \psi}(g)J_{\sigma^\prime, \psi^{-1}}(g) =
\left\{\begin{array}{ll}
\frac{|G_n(\mathfrak{f})|}{|U_n(\mathfrak{f})|\dim \sigma}, & \text{ if } \sigma^\prime \cong \tilde{\sigma} \\
0, & \text{ otherwise.}
\end{array}\right.
\end{equation*}
Here, $\tilde{\sigma}$ is the contragradient of $\sigma$.
\end{lemma}

\begin{proof}
For ease of notations, we set $\chi = \mathrm{tr}_\sigma$ and $\chi^\prime = \mathrm{tr}_{\sigma^\prime}$. Then $\overline{\chi}$ is the character of $\tilde{\sigma}$. Using \Cref{eqn:finite-bessel},
\begin{equation*}
\begin{split}
S &= \sum_{g \in U_n(\mathfrak{f}) \backslash G_n(\mathfrak{f})} J_{\sigma, \psi}(g)J_{\sigma^\prime, \psi^{-1}}(g) \\
&= \frac{1}{|U_n(\mathfrak{f})|^3} \sum_{g \in G_n(\mathfrak{f})} \left(\sum_{x \in U_n(\mathfrak{f})}\psi(x^{-1})\chi(gx)\right) \left(\sum_{y \in U_n(\mathfrak{f})}\psi^{-1}(y^{-1})\chi^\prime(gy)\right) \\
&= \frac{1}{|U_n(\mathfrak{f})|^3} \sum_x \sum_y \psi(x^{-1}y)\sum_g \chi(gx)\chi^\prime(gy)
\end{split}
\end{equation*}
With a change a variable $g$ to $gy^{-1}$,
\begin{equation*}
\begin{split}
S &= \frac{1}{|U_n(\mathfrak{f})|^3} \sum_x \sum_y \psi(x^{-1}y)\sum_g \chi(gy^{-1}x)\chi^\prime(g) \\
&= \frac{1}{|U_n(\mathfrak{f})|^3} \sum_x \sum_y \psi(x^{-1}y)\sum_g \overline{\chi}(x^{-1}yg^{-1})\chi^\prime(g)
\end{split}
\end{equation*}
By Theorem 2.13 of \cite{Isaacs06}, $\sum_g \overline{\chi}(x^{-1}yg^{-1})\chi^\prime(g)$ is nonzero if and only if $\overline{\chi} = \chi^\prime$. In this case, $\tilde{\sigma} \cong \sigma^\prime$ and
\begin{equation*}
\begin{split}
S &=  \frac{1}{|U_n(\mathfrak{f})|^3} \sum_x \sum_y \psi(x^{-1}y) |G_n(\mathfrak{f})|\frac{\chi^\prime(x^{-1}y)}{\dim \sigma} \\
&= \frac{|G_n(\mathfrak{f})|}{|U_n(\mathfrak{f})|^3 \dim \sigma} \sum_x \sum_y \psi(y)\chi^\prime(y) \\
&= \frac{|G_n(\mathfrak{f})|}{|U_n(\mathfrak{f})| \dim \sigma} J_{\sigma^\prime, \psi^{-1}}(1) \\
&= \frac{|G_n(\mathfrak{f})|}{|U_n(\mathfrak{f})| \dim \sigma}.
\end{split}
\end{equation*}
The second equality above comes from a change of variable from $y$ to $xy$.
\end{proof}

Putting \Cref{thm:GnXGm,thm:GnXGn} together, we have our main theorem, \Cref{thm:main}. Let level zero representation $\pi_i$ come from $(\lambda_i, \sigma_i)$, for $i = 1, 2$. If $\sigma_1 \not\cong \tilde{\sigma}_2$, then by \Cref{lem:level-zero-repn}, $L(s, \pi_1 \times \pi_2) = 1$. In this case,
$$\epsilon(s, \pi_1 \times \pi_2, \psi) = \gamma(s, \pi_1 \times \pi_2, \psi).$$
The equality in \Cref{thm:main} implies that $\epsilon(s, \pi_1 \times \pi_2, \psi)$ is a constant, i.e., the conductor of the pair $(\pi_1, \pi_2)$ is zero with respect to $\psi$. This is also a consequence of \cite[Theorem 6.5(i)]{BushnellHenniartKutzko98}. If $\sigma_1 \cong \tilde{\sigma}_2$, equality in our main theorem fails. However, from the proof of \Cref{thm:GnXGn} and \cite[Theorem 6.5(ii)]{BushnellHenniartKutzko98}, we have

\begin{corollary}
If $\sigma_1$ and $\sigma_2$ are irreducible cuspidal representations of $G_n(\mathfrak{f})$ such that $\sigma_1 \cong \tilde{\sigma}_2$, then
$$\gamma(\sigma_1 \times \sigma_2, \psi) = -1.$$
\end{corollary}

\begin{proof}
Let $\pi_i$ be a level zero representation of $G_n(F)$ from $(\lambda_i \in \C^\times, \sigma_i)$, for $i = 1, 2$. Let $\omega$ be the product of central characters of $\pi_1$ and $\pi_2$. Then $\omega(\varpi) = \lambda_1\lambda_2$, where $\varpi$ is a uniformizer of $F$. From \Cref{eqn:cal_of_gamma_GnXGn,eqn:finite-gamma-in-bessel-GnXGn}, we get
\begin{equation*}
\gamma(s, \pi_1 \times \pi_2, \psi) =  \frac{q^{n(s-1)}}{\lambda_1\lambda_2 - q^{n(s-1)}} \sum_{U_n(\mathfrak{f}) \backslash G_n(\mathfrak{f})} J_{\sigma_1, \psi}(g)J_{\sigma_2, \psi^{-1}}(g) + \gamma(\sigma_1 \times \sigma_2, \psi).
\end{equation*}
The dimension of $\sigma_1$ or $\sigma_2$ is $(q-1)(q^2-1)\cdots(q^{n-1}-1)$, see Corollary 1 of Theorem 2.3 in \cite{Gelfand70}. Thus, by \Cref{lem:representations-and-bessels},
\begin{equation}
\gamma(s, \pi_1 \times \pi_2, \psi) =  \frac{q^{n(s-1)}(q^n-1)}{\lambda_1\lambda_2 - q^{n(s-1)}} + \gamma(\sigma_1 \times \sigma_2, \psi).
\label{eqn:level-zero-gamma}
\end{equation}
Now from \Cref{cor:l-function},
$$L(s, \pi_1 \times \pi_2) = (1-\lambda_1\lambda_2 q^{-ns})^{-1},$$
and
$$L(s, \tilde{\pi}_1 \times \tilde{\pi}_2) = (1-\lambda_1^{-1}\lambda_2^{-1}q^{-ns})^{-1}.$$
Therefore,
\begin{equation}
\begin{split}
\epsilon(s, \pi_1 \times \pi_2, \psi) &= \frac{\gamma(s, \pi_1 \times \pi_2, \psi)L(s, \pi_1 \times \pi_2)}{L(1-s, \tilde{\pi}_1 \times \tilde{\pi}_2)} \\
&= \frac{(\lambda_1\lambda_2)^{-1}q^{-n}[q^n-1-\gamma(\sigma_1\times\sigma_2, \psi)]q^{ns}+\gamma(\sigma_1\times\sigma_2, \psi)}{q^{ns}-\lambda_1\lambda_2} \cdot q^{ns}.
\end{split}
\label{eqn:level-zero-epsilon}
\end{equation}
From \cite[Theorem 6.5(ii)]{BushnellHenniartKutzko98}, we deduce that the conductor of the pair $(\pi_1, \pi_2)$ with respect to $\psi$ is $n$, i.e., $\epsilon(s, \pi_1 \times \pi_2, \psi) = cq^{ns}$ for some constant $c \in \C^\times$. From \Cref{eqn:level-zero-epsilon}, we get
$$c=\frac{(\lambda_1\lambda_2)^{-1}q^{-n}[q^n-1-\gamma(\sigma_1\times\sigma_2, \psi)]q^{ns}+\gamma(\sigma_1\times\sigma_2, \psi)}{q^{ns}-\lambda_1\lambda_2}.$$
Clearing the denominator,
$$[c - (\lambda_1\lambda_2)^{-1}q^{-n}(q^n-1-\gamma(\sigma_1\times\sigma_2, \psi))]q^{ns} - (c\lambda_1\lambda_2+\gamma(\sigma_1\times\sigma_2, \psi))=0.$$
Since $c$ is a constant independent of $s$, we have
\begin{equation*}
\left\{\begin{array}{l}
c - (\lambda_1\lambda_2)^{-1}q^{-n}(q^n-1-\gamma(\sigma_1\times\sigma_2, \psi)) = 0 \\
c\lambda_1\lambda_2+\gamma(\sigma_1\times\sigma_2, \psi) = 0
\end{array}\right.
\end{equation*}
Treating $c$ and $\gamma(\sigma_1 \times \sigma_2, \psi)$ as unknowns, from the linear system of equations above, we get $c = (\lambda_1\lambda_2)^{-1}$ and $\gamma(\sigma_1 \times \sigma_2, \psi) = -1$.
\end{proof}

We remark that, putting $\gamma(\sigma_1 \times \sigma_2, \psi) = -1$ in \Cref{eqn:level-zero-gamma,eqn:level-zero-epsilon},
\begin{equation*}
\gamma(s, \pi_1 \times \pi_2, \psi) = -\frac{q^{ns}-\lambda_1\lambda_2}{q^{n(s-1)}-\lambda_1\lambda_2}
\text{ and }
\epsilon(s, \pi_1 \times \pi_2, \psi) = (\lambda_1\lambda_2)^{-1}q^{ns}.
\end{equation*}

\section{Local converse theorem}

As an application, we are going to prove a variant of the local converse theorem for level zero representations. In light of \Cref{thm:GnXGm}, we will need converse theorems over a finite field. Roditty proves in \cite{Roditty10} that an irreducible cuspidal representation $\sigma$ of $G_n(\mathfrak{f})$ is determined up to a central character by the family of gamma factors
$$\{\gamma(\sigma \times \tau, \psi)\},$$
where $\tau$ runs over all irreducible generic representations of $G_r(\mathfrak{f})$ for $1 \le r \le n-2$. Nien improves this result by shrinking the range of $r$ to $1 \le r \le \left[\frac{n}{2}\right]$ in \cite{Nien14}, which is an analogue of a Jacquet's conjecture. In order to prove our local converse theorem, we need a stronger version Nien's result: allow $\tau$ only to run over all irreducible cuspidal representations of $G_r(\mathfrak{f})$ for $1 \le r \le \left[\frac{n}{2}\right]$.

\begin{lemma}\label{lem:criteria-for-being-zero}
Let $H$ be a function on $G_t(\mathfrak{f})$ satisfying
$$H(ug) = \psi(u)H(g),$$
for all $u \in U_t(\mathfrak{f}), g \in G_t(\mathfrak{f})$, and
$$\sum_{u \in U^\prime}H(gux) = 0,$$
for any nontrivial standard unipotent subgroup $U^\prime \subset G_t(\mathfrak{f})$ and $x \in G_t(\mathfrak{f})$. If
$$\sum_{g \in U_t(\mathfrak{f})\backslash G_t(\mathfrak{f})} H(g)W_\tau(g) = 0$$
for all $W_\tau \in \mathcal{W}(\tau, \psi^{-1})$ with $\tau$ exhausting all irreducible cuspidal representations of $G_t(\mathfrak{f})$, then
$$H \equiv 0.$$
\end{lemma}

\begin{proof}
Consider the representation $V(H)$ of $G_t(\mathfrak{f})$ generated by $R(x)H$ for $x \in G_t(\mathfrak{f})$, where $R$ is the right multiplication. Since $H(ug) = \psi(u)H(g)$, $V(H) \subseteq \mathrm{Ind}_{U_n(\mathfrak{f})}^{G_n(\mathfrak{f})} \psi$. The assumption that 
$$\sum_{u \in U^\prime}H(gux) = 0,$$
for all nontrivial standard unipotent subgroup $U^\prime$ and $x \in G_t(\mathfrak{f})$ is equivalent to that for any $f \in V(H)$, and for any nontrivial standard unipotent subgroup $U^\prime$,
$$\sum_{u \in U^\prime} R(u)f = 0.$$
By \cite[Proposition 3.2 and Theorem 3.3]{Gelfand70}, we know that a representation $\sigma$ of $G_t(\mathfrak{f})$ is a sum of irreducible cuspidal representations if and only if, for any nontrivial standard unipotent subgroup $U^\prime$,
 $$\sum_{u \in U^\prime} \sigma(u) = 0.$$
Thus, $V(H)$ is a sum of irreducible cuspidal representations. Therefore,
$$H = \sum_{i=1}^{k} W_{\tau_i},$$
for some irreducible cuspidal representations $\tau_i$ and $W_{\tau_i} \in \mathcal{W}(\tau_i, \psi)$. Since the complex conjugate $\overline{W_{\tau_i}}$ is in $\mathcal{W}(\overline{\tau_i}, \psi^{-1})$, where $\overline{\tau}$ is the complex conjugate of $\tau$,
$$\sum_{g \in U_t(\mathfrak{f})\backslash G_t(\mathfrak{f})} H(g)\overline{W_{\tau_i}(g)} = 0$$
for $i = 1, \dots, k$, so we get
$$\sum_{g \in U_t(\mathfrak{f})\backslash G_t(\mathfrak{f})} H(g)\overline{H(g)} = 0.$$
This forces $H \equiv 0$.
\end{proof}

This lemma allows us to get a slightly stronger converse theorem over a finite field, which is a key to our local converse theorem for level zero representations.

\begin{theorem}\label{thm:local-converse-finite}
Let $\sigma_1$ and $\sigma_2$ be irreducible cuspidal representations of $G_n(\mathfrak{f})$ with the same central character. If
$$\gamma(\sigma_1 \times \tau, \psi) = \gamma(\sigma_2 \times \tau, \psi),$$
for all irreducible cuspidal representations $\tau$ of $G_r(\mathfrak{f})$ with $1 \le r \le \left[\frac{n}{2}\right]$, then $\sigma_1 \cong \sigma_2$.
\end{theorem}

\begin{proof}
Identical to the proof of \cite[Theorem 3.9]{Nien14}, only with her Lemma 3.1 replaced by our \Cref{lem:criteria-for-being-zero}. The key idea is to prove the Bessel functions $J_{\sigma_1, \psi}$ and $J_{\sigma_2, \psi}$ are equal. The equality of these two Bessel functions is established through induction on levels. To make her proof work, we must show that, for any $t \in G_n(\mathfrak{f})$, whenever
\begin{equation*}
\begin{split}
&\sum_{m \in U_r(\mathfrak{f}) \backslash G_r(\mathfrak{f})} J_{\sigma_1, \psi}
\left(\begin{pmatrix}0 & I_{n-r} \\ m & 0\end{pmatrix}t\right)
W_{\tau}(m) \\
=&\sum_{m \in U_r(\mathfrak{f}) \backslash G_r(\mathfrak{f})} J_{\sigma_2, \psi}
\left(\begin{pmatrix}0 & I_{n-r} \\ m & 0\end{pmatrix}t\right)
W_{\tau}(m)
\end{split}
\end{equation*}
for all $W_{\tau} \in \mathcal{W}(\tau, \psi^{-1})$ with $\tau$ exhausting all cuspidal representations of $G_r(\mathfrak{f})$, then
$$J_{\sigma_1, \psi} \left(\begin{pmatrix}0 & I_{n-r} \\ m & 0\end{pmatrix}t\right) = J_{\sigma_2, \psi} \left(\begin{pmatrix}0 & I_{n-r} \\ m & 0\end{pmatrix}t\right),$$
for all $m \in G_r(\mathfrak{f})$. If we define $H: G_r(\mathfrak{f}) \to \C$ by
$$H(m) = J_{\sigma_1, \psi}\left(\begin{pmatrix}0 & I_{n-r} \\ m & 0\end{pmatrix}t\right) - J_{\sigma_2, \psi}\left(\begin{pmatrix}0 & I_{n-r} \\ m & 0\end{pmatrix}t\right),$$
then we can easily get $H(um) = \psi(u)H(m)$, because $J_{\sigma_1, \psi}$ and $J_{\sigma_2, \psi}$ are Bessel functions, and thus for $i =1, 2$,
\begin{equation*}
J_{\sigma_i, \psi}\left(\begin{pmatrix}0 & I_{n-r} \\ um & 0\end{pmatrix}t\right)
= J_{\sigma_i, \psi}\left(\begin{pmatrix}I_{n-r} & 0 \\ 0 & u\end{pmatrix}\begin{pmatrix}0 & I_{n-1} \\ m & 0\end{pmatrix}t\right)
= \psi(u)J_{\sigma_i, \psi}\left(\begin{pmatrix}0 & I_{n-r} \\ m & 0\end{pmatrix}t\right).
\end{equation*}
Since $J_{\sigma_i, \psi}$ is in a cuspidal representation, so for any nontrivial standard unipotent subgroup $U^\prime \subset G_n(\mathfrak{f})$ and any $x \in G_n(\mathfrak{f})$, we have
\begin{equation*}
\sum_{u \in U^\prime} R(u)R(x)J_{\sigma_i, \psi} = 0,
\end{equation*}
i.e., for any $g \in G_n(\mathfrak{f})$,
\begin{equation*}
\sum_{u \in U^\prime} J_{\sigma_i, \psi}(gux) = 0.
\end{equation*}
In particular, for any nontrivial unipotent subgroup $U^\prime \subset G_r(\mathfrak{f}), x, m \in G_r(\mathfrak{f})$ and $p \in G_n(\mathfrak{f})$, we have
\begin{equation*}
\sum_{u \in U^\prime} J_{\sigma_i, \psi}
\left(
\begin{pmatrix} 0 & I_{n-r} \\ m & 0\end{pmatrix}
\begin{pmatrix} u & 0 \\ 0 & I_{n-r}\end{pmatrix}
\begin{pmatrix} x & 0 \\ 0 & I_{n-r}\end{pmatrix}
t
\right) = 0.
\end{equation*}
This equality will imply then for any nontrivial unipotent subgroup $U^\prime \subset G_r(\mathfrak{f})$ and $x, m \in G_r(\mathfrak{f})$,
$$\sum_{u \in U^\prime} H(mux) = 0.$$
By \Cref{lem:criteria-for-being-zero}, $H(m) = 0$ for any $m \in G_r(\mathfrak{f})$, i.e., 
$$J_{\sigma_1, \psi} \left(\begin{pmatrix}0 & I_{n-r} \\ m & 0\end{pmatrix}t\right) = J_{\sigma_2, \psi} \left(\begin{pmatrix}0 & I_{n-r} \\ m & 0\end{pmatrix}t\right),$$
for any $m \in G_r(\mathfrak{f})$.
\end{proof}

\begin{theorem}\label{thm:local-converse-level-zero}
Let $\pi_1$ and $\pi_2$ be level zero representations of $G_n(F)$ with the same central character. If
$$\gamma(s, \pi_1 \times \tau, \psi) = \gamma(s, \pi_2 \times \tau, \psi)$$
for all level zero representations of $G_r(F)$ with $1 \le r \le \left[\frac{n}{2}\right]$, then $\pi_1 \cong \pi_2$.
\end{theorem}

\begin{proof}
For $i = 1, 2$, we assume $\pi_i$ is from $\sigma_i$. That is to say, there exists a representation $\Lambda_i$ on $F^\times G_n(\mathfrak{o})$ such that it restricts to an inflation $\sigma_i$ on $G_n(\mathfrak{o})$ and 
$$\pi_i \cong \mathrm{ind}_{F^\times G_n(\mathfrak{o})}^{G_n(F)} \Lambda_i.$$
Thus, to prove $\pi_1 \cong \pi_2$, we just need to prove $\Lambda_1 \cong \Lambda_2$. Since by assumption $\pi_1$ and $\pi_2$ have the same central character, it is enough to prove $\sigma_1 \cong \sigma_2$.

Since the central character $\omega_{\pi_i}$ of $\pi_i$ restricts to that of $\sigma_i$, $\sigma_1$ and $\sigma_2$ also have the same central character. By assumption on $\gamma$-factors and \Cref{thm:GnXGm},
$$\gamma(\sigma_1 \times \tau, \psi) = \gamma(\sigma_2 \times \tau, \psi),$$
for all irreducible cuspidal representations $\tau$ of $G_r(\mathfrak{f})$ with $1 \le r \le \left[\frac{n}{2}\right]$. Applying \Cref{thm:local-converse-finite}, we have $\sigma_1 \cong \sigma_2$.
\end{proof}

\section*{Acknowledgement}

The author thanks his advisor, Prof. James Cogdell, for helpful comments and discussions. He also thanks Qing Zhang for pointing out references on local converse theorem over finite fields.

\bibliographystyle{abbrv}
\bibliography{references}

\end{document}